\documentclass[11pt,a4paper,reqno]{amsart}
\usepackage[english]{babel}
\usepackage[applemac]{inputenc}
\usepackage[T1]{fontenc}
\usepackage{palatino}
\usepackage{verbatim}
\usepackage{amsmath}
\usepackage{amssymb}
\usepackage{amsthm}
\usepackage{amsfonts}
\usepackage{graphicx}
\usepackage{mathtools}
\usepackage{enumitem}

\usepackage[colorlinks = true, citecolor = black]{hyperref}
\usepackage{tikz}
\usetikzlibrary{arrows}

\newcommand{\R}{\mathbb{R}}

\newcommand{\calD}{\mathcal{D}}
\newcommand{\calH}{\mathcal{H}}

\newcommand{\spt}{\operatorname{spt}}
\newcommand{\Hd}{\dim_{\mathrm{H}}}

\newcommand{\spa}{\operatorname{span}}

\newcommand{\dist}{\operatorname{dist}}

\newcommand{\sgn}{\operatorname{sgn}}

\newcommand{\calC}{\mathcal{C}}

\newcommand{\B}{\mathbf{B}}

\numberwithin{equation}{section}

\theoremstyle{plain}
\newtheorem{thm}{Theorem}[section]

\newtheorem{lemma}[thm]{Lemma}

\theoremstyle{definition}

\newtheorem{notation}[thm]{Notation}

\theoremstyle{remark}

\newcommand{\nref}[1]{(\hyperref[#1]{#1})}

\begin{document}

\pagestyle{headings}

\title[Restricted families of projections to planes in $\R^{3}$]{Improved bounds for \\ restricted families of projections to planes in $\R^{3}$}

\author{Tuomas Orponen}
\author{Laura Venieri}
\address{Department of Mathematics and Statistics \\
        P.O.\ Box 68 (Gustaf H\"allstr\"omin katu 2b) \\
        FI-00014 University of Helsinki \\
        Finland}
\email{tuomas.orponen@helsinki.fi}
\email{laura.venieri@helsinki.fi}

\thanks{T.O. and L.V. are supported by the Academy of Finland via the project \emph{Quantitative rectifiability in Euclidean and non-Euclidean spaces}, grant numbers 309365 and 314172}
\date{\today}

\begin{abstract} For $e \in S^{2}$, the unit sphere in $\R^3$, let $\pi_{e}$ be the orthogonal projection to $e^{\perp} \subset \R^{3}$, and let $W \subset \R^{3}$ be any $2$-plane, which is not a subspace. We prove that if $K \subset \R^{3}$ is a Borel set with $\Hd K \leq \tfrac{3}{2}$, then $\Hd \pi_{e}(K) = \Hd K$ for $\calH^{1}$ almost every $e \in S^{2} \cap W$, where $\calH^{1}$ denotes the $1$-dimensional Hausdorff measure and $\Hd$ the Hausdorff dimension. This was known earlier, due to J\"arvenp\"a\"a, J\"arvenp\"a\"a, Ledrappier and Leikas, for Borel sets $K$ with $\Hd K \leq 1$. We also prove a partial result for sets with dimension exceeding $3/2$, improving earlier bounds by D. Oberlin and R. Oberlin. \end{abstract}

\maketitle

\section{Introduction}

How well is Hausdorff dimension preserved by orthogonal projections to planes in $\R^{3}$? This paper is a sequel to \cite{KOV}, which considered the same question for projections to lines. Given $e \in S^{2}$, write $\rho_{e} \colon \R^{3} \to \ell_{e}$ and $\pi_{e} \colon \R^{3} \to V_{e}$ for the orthogonal projections to $\ell_{e} := \spa(e)$ and $V_{e} := e^{\perp}$, respectively. The fundamental result in the area is due to Marstrand \cite{Mar} and Mattila \cite{Ma}:  if $K \subset \R^{3}$ is a Borel set, then
\begin{itemize}
\item[(MM1)] $\Hd \rho_{e}(K) = \min\{1,\Hd K\}$ for $\calH^{2}$ almost every $e \in S^{2}$, and
\item[(MM2)] $\Hd \pi_{e}(K) = \min\{2,\Hd K\}$ for $\calH^{2}$ almost every $e \in S^{2}$.
\end{itemize}
Recent evidence suggests that, in (MM1)--(MM2), the $2$-dimensional measure $\calH^{2}$ on $S^{2}$ can be replaced by length measure on certain curves $\Gamma \subset S^{2}$. The main result in \cite{KOV} proved this for part (MM1), whenever $\Gamma$ is a circle, but not a great circle (the great circles are a "degenerate" case, having non-trivial orthogonal complement). We refer the reader to \cite{KOV} for a broader introduction, and earlier results, on the projections $\rho_{e}$. In this paper, we consider part (MM2) in the same setting:
\begin{thm}\label{main} Let $W \subset \R^{3}$ be a $2$-plane, which is not a subspace. If $K \subset \R^{3}$ is a Borel set, then
\begin{displaymath} \Hd \pi_{e}(K) \geq \min\left\{\Hd K,1 + \tfrac{\Hd K}{3} \right\} \end{displaymath}
for $\calH^{1}$ almost every $e$ on the circle $S_{W} = S^{2} \cap W$. In particular, the projections $\pi_{e}$, $e \in S_{W}$, preserve $\calH^{1}$ almost surely the dimension of at most $\tfrac{3}{2}$-dimensional Borel sets. 
\end{thm}

\subsection{Previous results} For Borel sets $K \subset \R^{3}$ with $\Hd K \leq 1$, Theorem \ref{main} follows from the "potential theoretic method" of Kaufman \cite{Ka}. Then $\Hd \pi_{e}(K) = \Hd K$ for $\calH^{1}$ almost every $e \in S_{W}$, \textbf{even if the $2$-plane $W \subset \R^{3}$ is a subspace}. This result is due to J\"arvenp\"a\"a, J\"arvenp\"a\"a, Ledrappier and Leikas \cite[Theorem 3.2]{JJLL}, and a proof is also given in \cite[Proposition 1.5]{FO}. The requirement that $W$ is not a subspace only becomes necessary when $\Hd K > 1$: to see this, consider $K = W := \R^{2} \times \{0\}$, which has $\Hd K = 2$, yet
\begin{equation}\label{counterEx} \Hd \pi_{e}(K) = 1, \qquad e \in S_{W}. \end{equation}

When $W$ is not a subspace, the paper \cite{O} of the first author (building on the ideas developed in collaboration with F\"assler in \cite{FO}) found an $\epsilon$-improvement over the J\"arvenp\"a\"a-J\"arvenp\"a\"a-Ledrappier-Leikas bound: if $\Hd K = s > 1$, there exists $\epsilon(s) > 0$ such that $\Hd K \geq 1 + \epsilon(s)$ for $\calH^{1}$ almost every $e \in S_{W}$. Around the same time, D. Oberlin and R. Oberlin \cite{OO} applied Fourier restriction theory to obtain the following estimates: if $K \subset \R^{3}$ is a Borel set with $\Hd K =: s$, then
\begin{equation}\label{oberlinIneq} \Hd \pi_{e}(K) \geq \begin{cases} \tfrac{3s}{4}, & \text{if } 1 \leq s \leq 2,\\ \min\{s - \tfrac{1}{2},2\}, & \text{if } 2 \leq s \leq 3, \end{cases} \end{equation}
for $\calH^{1}$ almost every $e \in S_{W}$. The results in \cite{FO} and \cite{OO} also apply to a more general class of $\calC^{2}$-curves, of which the circles $S_{W}$ are a basic example.

In summary, prior to the current paper, the result in \cite{O} was the record for $1 \leq \Hd K \leq \tfrac{4}{3} + \epsilon$, and \eqref{oberlinIneq} superseded it once $\Hd K > \tfrac{4}{3} + \epsilon$. Theorem \ref{main} improves on both results for $1 < \Hd K \leq \tfrac{3}{2}$ (being sharp in that range), and improves on \eqref{oberlinIneq} whenever $\Hd K < \tfrac{9}{4}$. 

As a related development, we mention the recent paper of Chen \cite{Ch}, where the author constructs, for any $\alpha \in (1,2]$, an $\alpha$-Ahlfors-David set $G \subset S^{2}$ such that (MM1)--(MM2) are valid for $\calH^{\alpha}|_{G}$ in place of $\calH^{2}$. It seems likely that Chen's method also works with $\alpha = 1$, if Ahlfors-David regularity is relaxed to $0 < \calH^{1}(G) < \infty$, but the resulting set $G$ needs to be much more "uniformly distributed" than the circles $S_{W}$ (or, in fact, any other curves $\Gamma \subset S^{2}$ of finite length), see \cite[Lemmas 2.1-2.2]{Ch}.

\subsection{A few words on the proof}\label{outline} The arguments in the current paper are similar to those in \cite{KOV}, but there is a natural reason why \cite{KOV} was written first. The paper \cite{KOV} started with the observation that problems concerning the $1$-dimensional projections $\rho_{e}$ could be transformed into those concerning incidences between certain plane curves, namely sine waves; this operation is explained after the statement of \cite[Theorem 1.5]{KOV}. Then, the incidence problem for sine waves was solved using techniques developed by Wolff \cite{Wo2,Wo3} in his fundamental study of circular Kakeya problems and local smoothing estimates.

One could execute a similar strategy with the projections $\pi_{e}$, but one would end up with an incidence problem for vertical helices in $\R^{3}$ (we thank Tam\'as Keleti for pointing this out). At the time of writing \cite{KOV}, this problem seemed much harder than the one about sine waves. In retrospect, it turns out that attempting such a transformation only causes complications in the case of the projections $\pi_{e}$: instead, one should observe that any "incidence" of the form $\pi_{e}(z) = \pi_{e}(z')$, for some $e \in S_{W}$ and $z,z' \in \R^{3}$, translates directly to the tangency of a pair of planar circles, corresponding to to $z,z'$. Then, one can start finding upper and lower bounds on the number of such tangencies. Theorem \ref{main} follows from this approach.

\subsection{Why $3/2$?} Theorem \ref{main} is sharp for up to $3/2$-dimensional sets: where does the threshold come from? As we explained in the previous paragraph, an "incidence" of the form $\pi_{e}(z) = \pi_{e}(z')$ can be viewed as a tangency between a pair of planar circles. This way, the problem of finding lower bounds for the size of $\pi_{e}(K)$ can be translated into a "tangency counting problem": given a family of planar circles, how many pairwise (approximate) tangencies can occur between them? As far as we know, the strongest available result on the tangency counting problem is due to Wolff, \cite[Lemma 1.4]{Wo3}. This result was a key component in \cite{KOV}. Also in the current paper, Wolff's lemma could be used to give a heuristic justification of Theorem \ref{main}. In a remark on \cite[p. 1254]{Wo3}, Wolff writes that \cite[Lemma 1.4]{Wo3} is unlikely to be sharp, and proposes an optimal exponent. With the conjectured exponent in hand, the heuristic argument mentioned above would yield the sharp version of Theorem \ref{main}. 

Despite some effort, we were unable to make the heuristic argument rigorous, so we will not even attempt to give any details. Instead, our proof of Theorem \ref{main} is quite elementary and self-contained, avoiding the use of \cite[Lemma 1.4]{Wo3} altogether. It  seems that the elementary approach cannot be pushed further, so topping the $3/2$-bound in Theorem \ref{main} would likely entail making progress in Wolff's tangency-counting problem.

\begin{notation} We use the standard notation $A \lesssim_{p} B$ if there exists a constant $C \geq 1$, depending only on the parameter $p$, such that $A \leq CB$. Self-explanatory variants include $A \gtrsim_{p} B$ and $A \sim_{p} B$. We will also use the not-so standard notation $A \lessapprox B$, which will be explained in Notation \ref{lessapproxNotation}.

A closed ball with centre $p \in \R^{d}$ and radius $r > 0$ will be denoted by $B(p,r)$. The letter $z$ will denote a point in $\R^{3}$, the letters $x,y$ will denote points in $\R^{2}$.

\end{notation}

\section{Acknowledgements}

We are grateful to the referees for reading the manuscript carefully, and for making many useful suggestions. 

\section{Proof of the main theorem}\label{mainProof}

\subsection{Preliminary reductions, geometric observations, and notation} To prove Theorem \ref{main}, if suffices to consider the $2$-plane
\begin{displaymath} W := W_{1/\sqrt{2}} := \{(x,y,r) \in \R^{3} : r = \tfrac{1}{\sqrt{2}}\}. \end{displaymath}
In other words, if the lower bound in Theorem \ref{main} -- or any other bound, in fact -- is known for the projections $\pi_{e}$, $e \in S_{W}$, and all Borel sets $K \subset \R^{3}$, then the same bound follows for the projections $\pi_{e}$, $e \in W'$, for any other non-subspace $2$-plane $W' \subset \R^{2}$, and for all Borel sets $K \subset \R^{3}$. To see this, there are a couple of cases to consider. First, if $\dist(W',\{0\}) \geq 1$, there is nothing to prove, since $\calH^{1}(S_{W'}) = 0$. Also, if $W'$ is a rotated copy of $W$, then it is easy to see that the projections $\pi_{e}(K)$, $e \in S_{W'}$, are isometric to the projections $\pi_{e}(OK)$, $e \in S_{W}$, where $O \colon \R^{3} \to \R^{3}$ is a rotation. The (remaining) case of $W' = W_{t} = \{(x,y,r) : r = t\}$ for some $t \in (-1,1) \setminus \{0,\tfrac{1}{\sqrt{2}}\}$ takes a bit of linear algebra, namely the following lemma:
\begin{lemma} Let $t \in (-1,1) \setminus \{0\}$, parametrise $S_{W_{t}}$ by the curve $\gamma_{t} \colon [0,2\pi) \to S^{2}$,
\begin{equation}\label{form42} \gamma_{t}(\theta) = (\sqrt{1 - t^{2}}\cos \theta, \sqrt{1 - t^{2}} \sin \theta, t), \end{equation}
and write $V_{\theta}^{t} := \spa(\gamma_{t}(\theta)^{\perp})$. Write 
\begin{displaymath} \pi_{\theta}^{t} := \pi_{V_{\theta}^{t}} \quad \text{and} \quad \pi_{\theta} := \pi_{V^{1/\sqrt{2}}_{\theta}}, \end{displaymath}
so that $\pi_{\theta}^{t}$ and $\pi_{\theta}$ parametrise the projections related to $S_{W_{t}}$ and $S_{W}$, respectively. Then, there exists an invertible linear map $B_{t} \colon \R^{3} \to \R^{3}$, depending only on $t$, and a family of invertible linear maps $A_{\theta}^{t} \colon V_{\theta} \to V_{\theta}^{t}$, depending on both $\theta$ and $t$, such that the following relation holds for all $\theta \in [0,2\pi)$
\begin{displaymath} \pi_{\theta}^{t} = A_{\theta}^{t} \circ \pi_{\theta} \circ B_{t}. \end{displaymath} 
\end{lemma}
We postpone the proof of the lemma to Appendix \ref{appendixA}. The lemma is used as follows. Assume that we already know that Theorem \ref{main} holds for the circle $S_{W}$. With the notation of the previous lemma, this implies that for every Borel set $K \subset \R^{3}$, the following holds for almost every $\theta \in [0,2\pi)$:
\begin{equation}\label{form41} \Hd \pi_{\theta}(K) \geq \min\left\{\Hd K, 1 + \frac{\Hd K}{3}\right\} =: f(\Hd K). \end{equation}
Fixing $t \in (-1,1) \setminus \{0\}$, and applying \eqref{form41} to the set $B_{t}(K)$ yields
\begin{displaymath} \Hd \pi_{\theta}^{t}(K) = \Hd A_{\theta}^{t}[\pi_{\theta}(B_{t}(K))] = \Hd \pi_{\theta}(B_{t}(K)) \geq f(\Hd B_{t}(K)) = f(\Hd K) \end{displaymath}
for almost every $\theta \in [0,2\pi)$, using the fact that $A_{\theta}^{t}$ preserves the dimension of subsets of $V_{\theta}$, and $B_{t}$ preserves the dimension of subsets of $\R^{3}$. Thus, Theorem \ref{main} holds for the circle $S_{W_{t}}$ as well.

So, for the rest of the paper, we concentrate on the circle $S = S_{W} = S^{2} \cap W$. Note that the parametrisation $\gamma = \gamma_{1/\sqrt{2}} \colon [0,2\pi) \to S$ from \eqref{form42} is simply
\begin{displaymath} \gamma(\theta) = \tfrac{1}{\sqrt{2}}(\cos \theta, \sin \theta, 1). \end{displaymath}
As above, we write $V_{\theta} := \gamma(\theta)^{\perp}$ and $\pi_{\theta} := \pi_{\gamma(\theta)}$. We define the following "standard region"
\begin{displaymath} \B_{0} := \{z = (x,r) \in \R^{2} \times \R : r \in [\tfrac{1}{2},1] \text{ and } |x| \leq \tfrac{1}{4}\} \subset \R^{3}, \end{displaymath}
and we will always assume that the Borel set $K \subset \R^{3}$ in Theorem \ref{main} is contained in $\B_{0}$; this can be achieved by scalings and translations, which do not affect the dimension of the projections.

Given $z = (x,r),z' = (x',r') \in \R^{3}$, write
\begin{displaymath} \Delta(z,z') := ||x - x'| - |r - r'||. \end{displaymath}
This quantity measures the level of tangency between the planar circles $S(x,r)$ and $S(x',r')$ with centers $x$, $x'$ and radii $r$, $r'$ respectively. In particular, $\Delta(z,z') = 0$ if and only if the circles $S(x,r),S(x',r')$ are internally tangent. The motivation for this quantity, in the current paper, is the observation that if $z,z' \in \R^{3}$ are such that $|\pi_{\theta}(z) - \pi_{\theta}(z')| \leq \delta$ for some $\theta \in [0,2\pi)$, then $\Delta(z,z') \leq 2\delta$. Indeed, the condition $|\pi_{\theta}(z) - \pi_{\theta}(z')| \leq \delta$ is equivalent to $\dist(z - z',\spa(\gamma(\theta))) \leq \delta$. Hence, there is a point $(y,s) = \alpha(\cos \theta, \sin \theta,1) \in \spa(\gamma(\theta))$ satisfying $|(z - z') - (y,s)| \leq \delta$. This implies, using $|y| = |s|$, that
\begin{equation}\label{form40} \Delta(z,z') = ||x - x'| - |r - r'|| \leq ||x - x'| - |y|| + ||r - r'| - |s|| \leq 2\delta, \end{equation}
as claimed.

We now introduce a "multiplicity function" $m^{\delta}_{\pi}$. Given a finite measure $\mu$ on $\B_{0}$, a parameter $\theta \in [0,2\pi)$, and a scale $\delta > 0$, write
\begin{displaymath} m^{\delta}_{\pi}(\pi_{\theta}(z)) := \mu(\{z' \in \R^{3} : |\pi_{\theta}(z) - \pi_{\theta}(z')| \leq \delta\}), \qquad z \in \R^{3}. \end{displaymath} 
In other words, $m^{\delta}_{\pi}(\pi_{\theta}(z))$ is the $\mu$ measure of the $2\delta$-tube $\pi_{\theta}^{-1}(B(\pi_{\theta}(z),\delta)) \subset \R^{3}$; the measure $\mu$ is always "fixed" in applications, so we suppress it from the notation. Informally, if the measure $\mu$ is $s$-dimensional, then the "expected" value of $m_{\pi}^{\delta}(\pi_{\theta}(z))$ is about $\delta^{s}$, for $z \in \spt \mu$; any values significantly larger should be interpreted as "overlap" in the projection $\pi_{\theta}(\spt \mu)$, at scale $\delta$. In the proofs below, the multiplicity function $m^{\delta}_{\pi}$ will often be "restricted" as follows: if $B \subset \B_{0}$ is any set, we write
\begin{displaymath} m^{\delta}_{\pi}(\pi_{\theta}(z)|B) := \mu(\{z' \in B : |\pi_{\theta}(z) - \pi_{\theta}(z')| \leq \delta\}). \end{displaymath}
In practice, $B$ will often have the form $B = B_{\delta}(z) = \{z' \in \B_{0} : \Delta(z,z') \leq 2\delta\}$, or
\begin{displaymath} B = B_{\delta,t}(z) = \{z' = (x',r') \in \B_{0} : \Delta(z,z') \leq 2\delta \text{ and } |z - z'| \in [t,2t]\} \end{displaymath}
or
\begin{displaymath} B = B_{\delta,\leq t}(z) = \{z' = (x',r') \in \B_{0} : \Delta(z,z') \leq 2\delta \text{ and } |z - z'| \leq 2t\}, \end{displaymath}
for some fixed vector $z = (x,r) \in \B_{0}$, and a dyadic number $t \in [\delta,1]$. Note that
\begin{equation}\label{form14} m^{\delta}_{\pi}(\pi_{\theta}(z)) = m^{\delta}_{\pi}(\pi_{\theta}(z)|B_{\delta}(z)) \end{equation} 
for $z \in \B_{0}$ and $\theta \in [0,2\pi)$, because $|\pi_{\theta}(z) - \pi_{\theta}(z')| \leq \delta$ implies $\Delta(z,z') \leq 2\delta$, as discussed above \eqref{form40}.

We record the following easy geometric fact:
\begin{lemma}\label{geoLemma} For $\delta > 0$, and distinct points $z,z' \in \R^{3}$, the set 
\begin{equation}\label{subLevelSet} E_{\delta}(z,z') := \{\theta \in [0,2\pi) : |\pi_{\theta}(z) - \pi_{\theta}(z')| \leq \delta\} \end{equation}
is contained in a single interval of length $\lesssim \min\{1,\delta/|z - z'|\}$. \end{lemma}

\begin{proof} It is easy to reduce to the case $z' = 0$ and $|z| = 1$, using the linearity of the projections $\pi_{\theta}$. Then, assume that $\theta_{1},\theta_{2} \in [0,2\pi)$ are such that
\begin{displaymath} |\pi_{\theta_{1}}(z)| \leq \delta \quad \text{and} \quad |\pi_{\theta_{2}}(z)| \leq \delta. \end{displaymath}
As discussed above, this implies that $z$ is at distance $\leq \delta$ from both the lines $\ell_{\theta_{1}} = \spa(\gamma(\theta_{1}))$ and $\ell_{\theta_{2}} = \spa(\gamma(\theta_{2}))$, hence there exist $\alpha_{1},\alpha_{2} \in \R$ such that
\begin{displaymath} |z - \alpha_{1}(\cos \theta_{1},\sin \theta_{1},1)| \leq \delta \quad \text{and} \quad |z - \alpha_{2}(\cos \theta_{2},\sin \theta_{2},1)| \leq \delta. \end{displaymath}
Since $|z| = 1$, this implies that $||\alpha_{1}| - 1/\sqrt{2}| \leq \delta$ and $||\alpha_{2}| - 1/\sqrt{2}| \leq \delta$, and consequently,
\begin{displaymath} |(\cos \theta_{1},\sin \theta_{1},1) - (\cos \theta_{2},\sin \theta_{2},1)| \lesssim \delta. \end{displaymath}
This yields $|\theta_{1} - \theta_{2}| \lesssim \delta = \delta/|z|$, as claimed. \end{proof}

\subsection{Estimating the multiplicity function $m_{\pi}^{\delta}$}  Now, we start proving Theorem \ref{main}. The strategy, as in \cite{KOV}, is to establish, first, an estimate (Lemma \ref{mainLemma}) for the multiplicity function $m_{\pi}^{\delta}$ at a fixed scale $\delta$, and for all Frostman-type measures $\mu$ on $\B_{0}$. Then, by a fairly abstract and standard procedure (Lemma \ref{multToDimension}), one can infer a lower bound for the dimension of the projections $\pi_{e}(K)$, $e \in S_{W}$.

\begin{lemma}\label{mainLemma} Fix $0 < s \leq 3$, and let $\mu$ be a probability measure on $\B_{0}$ satisfying the Frostman bound $\mu(B(z,r)) \leq C_{F}r^{s}$ for all $z \in \R^{3}$ and $r > 0$. Fix 
\begin{displaymath} \kappa > \max\left\{0,\frac{2s}{3} - 1\right\}. \end{displaymath}
Then, there exist constants $\delta(C_{F},\kappa,s) > 0$, and $\eta = \eta(\kappa,s) > 0$ such that the set $Z = Z_{\delta}$ of points $z \in \B_{0}$ satisfying
\begin{equation}\label{form2a} \calH^{1}(\{\theta \in [0,2\pi) : m^{\delta}_{\pi}(\pi_{\theta}(z)) \geq \delta^{s-\kappa}\}) \geq \delta^{\eta} \end{equation}
has measure $\mu(Z) \leq \delta^{\eta}$ for $0 < \delta \leq \delta(C_{F},\kappa,s)$.
\end{lemma}

\begin{notation}\label{lessapproxNotation} Below, the notation $A \lessapprox B$ means that there is an absolute constant $C \geq 1$ such that $A \leq C\log^{C}(1/\delta)B$. The notation $A \gtrapprox B$ means that $B \lessapprox A$, and $A \approx B$ means that $A \lessapprox B \lessapprox A$.
\end{notation}

\begin{proof}[Proof of Lemma \ref{mainLemma}] Write
\begin{displaymath} H(z) := \{\theta \in [0,2\pi) : m_{\pi}^{\delta}(\pi_{\theta}(z)) \geq \delta^{s-\kappa}\} \end{displaymath}
for $z \in \B_{0}$, so that $\calH^{1}(H(z)) \geq \delta^{\eta}$ for $z \in Z$. Assuming that
\begin{displaymath} \mu(Z) \geq \delta^{\eta}, \end{displaymath}
the task is to find a lower bound for $\eta$ (depending on $\kappa,s$). Fix $z \in Z$ and $\theta \in H(z)$. Then $m_{\pi}^{\delta}(\pi_{\theta}(z)) \geq \delta^{s - \kappa}$, and we claim that there exists a dyadic number $t = t(\theta,z) \in [\delta^{1 - 3\eta},1]$ such that
\begin{displaymath} m^{\delta}_{\pi}(\pi_{\theta}(z)|B_{\delta,t}(z)) \geq m \end{displaymath}
for some 
\begin{equation}\label{form5} m \gtrapprox \delta^{s - \kappa}. \end{equation}
Indeed, the existence of $t$ follows from the estimate
\begin{equation}\label{form22a} m^{\delta}_{\pi}(\pi_{\theta}(z)) \stackrel{\eqref{form14}}{=} m^{\delta}_{\pi}(\pi_{\theta}(z)|B_{\delta}(z)) \leq \mu(B(z,2\delta^{1 - 3\eta})) + \sum_{t \in [\delta^{1 - 3\eta},1]} m^{\delta}_{\pi}(\pi_{\theta}(z)|B_{\delta,t}(z)), \end{equation}
where the last sum runs over $\approx 1$ dyadic numbers $t \in [\delta^{1 - 3\eta},1]$. Next, by a few applications of the pigeonhole principle, and noting that 
\begin{displaymath} \mu(B(z,2\delta^{1 - 3\eta})) \lesssim \delta^{(1 - 3\eta)s} \end{displaymath}
is much smaller than $m^{\delta}_{\pi}(\pi_{\theta}(z)) \geq \delta^{s - \kappa}$ if $\eta > 0$ is small enough, the parameter $t$ can be "frozen": there exists a fixed dyadic number $t \in [\delta^{1 - 3\eta},1]$ (independent of $z,\theta$), and a subset $Z' \subset Z$ with $\mu(Z') \gtrapprox \mu(Z)$, such that
\begin{displaymath} \calH^{1}(H'(z)) \gtrapprox \delta^{\eta}, \qquad z \in Z', \end{displaymath}
where
\begin{displaymath} H'(z) := \{\theta \in [0,2\pi) : m^{\delta}_{\pi}(\pi_{\theta}(z)|B_{\delta,t}(z)) \geq m\}. \end{displaymath}
We abbreviate $B_{\delta,t}(z) =: B(z)$ in the sequel. For every $z \in Z'$, we construct further three subsets $H_{1}'(z),H_{2}'(z),H_{3}'(z) \subset H'(z)$ with the properties
\begin{equation}\label{form3a} \calH^{1}(H_{j}'(z)) \gtrapprox \delta^{2\eta} \quad \text{and} \quad \min_{1 \leq i < j \leq 3} \dist(H_{i}'(z),H_{j}'(z)) \gtrapprox \delta^{\eta}. \end{equation}
This is easily done by first covering $H'(z)$ by intervals of length $\calH^{1}(H'(z))/100 \gtrapprox \delta^{\eta}$, then picking five of them, which contain as much of $H'(z)$ as possible, and finally choosing the first, the third, and the fifth among them. For $j \in \{1,2,3\}$ fixed, define the relation
\begin{displaymath} z' \asymp_{j} z \quad \Longleftrightarrow \quad |z - z'| \in [t,2t) \text{ and } |\pi_{\theta}(z) - \pi_{\theta}(z')| \leq \delta \text{ for some } \theta \in H_{j}'(z),\end{displaymath} 
for $z = (x,r), z' = (x',r') \in \B_{0}$. We claim that, for $z \in Z'$ and $j \in \{1,2,3\}$ still fixed, there exist many points $z' \in \B_{0}$ with $z' \asymp_{j} z$, see \eqref{form6a} for a more precise statement.

To reach \eqref{form6a}, first cover $H_{j}'(z)$ by intervals $J_{j}^{1},\ldots,J_{j}^{N}$ with disjoint interiors and of length $C(\delta/t)$, where $C \geq 1$ is a suitable constant. Then, 
\begin{equation}\label{form8a} N \gtrsim \frac{\calH^{1}(H_{j}'(z))}{\delta/t} \gtrapprox t \cdot \delta^{2\eta - 1} \end{equation} 
by the first part of \eqref{form3a}. Note that $N \geq 1$ for small enough $\delta$, since $t \geq \delta^{1 - 3\eta}$. One may assume that each intersection $H_{j}'(z) \cap J_{j}^{i}$ contains a point $\theta_{j}^{i}(z) \in H_{j}'(z)$. Finally, discard at most half of the points $\theta_{j}^{i}(z)$, so that the separation between the remaining points is at least $C(\delta/t)$.

For $j \in \{1,2,3\}$ fixed, and $1 \leq i \leq N$, abbreviate $\theta_{i} := \theta_{j}^{i}(z) \in H_{j}'(z)$, and consider the points $z' = (x',r') \in \B_{0}$ such that $|z' - z| \in [t,2t)$ and $|\pi_{\theta_{i}}(z) - \pi_{\theta_{i}}(z')| \leq \delta$. Denote the set of these points by $A_{j}^{i}(z)$. Clearly $A_{j}^{i}(z) \subset \{z' \in \B_{0} : z' \asymp_{j} z\}$, and, by the definitions of $m_{\pi}^{\delta}(\pi_{\theta_{i}}(z)|B(z))$, $H'(z)$, and \eqref{form5},
\begin{equation}\label{form7a} \mu(A_{j}^{i}(z)) = m^{\delta}_{\pi}(\pi_{\theta_{i}}(z)|B(z)) \geq m \gtrapprox \delta^{s - \kappa}. \end{equation}
Moreover, the sets $A_{j}^{i_{1}}(z),A_{j}^{i_{2}}(z)$ are disjoint for $1 \leq i_{1} < i_{2} \leq N$: indeed, if $z' \in A_{j}^{i_{1}}(z) \cap A_{j}^{i_{2}}(z)$, then recall that $|\theta_{i_{1}} - \theta_{i_{2}}| \geq C(\delta/t)$, and note that
\begin{displaymath} |\pi_{\theta_{i_{1}}}(z) - \pi_{\theta_{i_{1}}}(z')| \leq \delta \quad \text{and} \quad |\pi_{\theta_{i_{2}}}(z) - \pi_{\theta_{i_{2}}}(z')| \leq \delta, \end{displaymath}
by the definition of $z' \in A_{j}^{i_{1}}(z) \cap A_{j}^{i_{2}}(z)$. This would mean that $\theta_{i_{1}},\theta_{i_{2}} \in E_{\delta}(z,z')$, which is impossible if $C \geq 1$ is large enough by $|z - z'| \sim t$, $|\theta_{i_{1}} - \theta_{i_{2}}| \geq C(\delta/t)$, and Lemma \ref{geoLemma}.

It follows from the disjointness of the sets $A_{j}^{i}(z)$, $1 \leq i \leq N$, and \eqref{form8a}-\eqref{form7a}, that
\begin{align} \mu(\{z' \in \B_{0} : z' \asymp_{j} z\}) & \geq \sum_{i = 1}^{N} \mu(A_{j}^{i}(z)) \notag\\
&\label{form6a} \gtrapprox N \cdot \delta^{s - \kappa} \gtrapprox t \delta^{2\eta + s - \kappa - 1} \end{align}
for $z \in Z'$ and $j \in \{1,2,3\}$. Consequently,
\begin{equation}\label{form4a} \mu^{4}(\{(z,z_{1},z_{2},z_{3}) \in Z' \times \B_{0}^{3} : z_{j} \asymp_{j} z\}) \gtrapprox t^{3} \delta^{3(2\eta + s - \kappa - 1)}\mu(Z') \gtrapprox t^{3} \cdot \delta^{7\eta + 3(s - \kappa - 1)}. \end{equation}
The rest of the argument is devoted to finding an upper bound for the left hand side of \eqref{form4a}; comparing the bounds will complete the proof. 

Fix $z_{1},z_{2},z_{3} \in \B_{0}$ such that there is at least one $z \in Z'$ with the property that $z_{j} \asymp_{j} z$ for all $1 \leq j \leq 3$. This implies that $|z_{j} - z| \sim t$ for all $1 \leq j \leq 3$, and in particular 
\begin{equation}\label{form9a} |z_{2} - z_{1}| \lesssim t \quad \text{and} \quad |z_{3} - z_{1}| \lesssim t. \end{equation}
The next question is: how many points $z \in Z'$ can there be, as above? Fix $z = (x,r) \in Z'$ such that $z_{j} \asymp_{j} z$ for $1 \leq j \leq 3$, and write $z_{j} := (x_{j},r_{j})$. Recall $Z' \subset \B_{0}$, so $|x|,|x_{j}| \leq \tfrac{1}{4}$, and $r,r_{j} \in [\tfrac{1}{2},1]$. By definition of $z \asymp_{j} z_{j}$, there exists $\theta_{j} = \theta(z,z_{j}) \in H_{j}'(z)$ such that 
\begin{displaymath} |\pi_{\theta_{j}}(z) - \pi_{\theta_{j}}(z_{j})| \leq \delta, \end{displaymath}
which implies that $\Delta(z , z') \leq 2\delta$ (recall \eqref{form40}). Further, we claim that
\begin{equation}\label{form11a} |e(z,z_{j}) + (\cos \theta_{j},\sin \theta_{j})| \lesssim \frac{\delta}{t}, \qquad j \in \{1,2,3\}, \end{equation}
where 
\begin{displaymath} e(z,z_{j}) := \sgn(r - r_j)\frac{x_{j} - x}{|x_{j} - x|} \in S^{1}. \end{displaymath}
We write
\begin{displaymath} \sigma := \sgn(r - r_{j}) \in \{-1,1\} \end{displaymath}
for brevity. The geometric meaning of the vector $e(z,z_{j})$ is that the circles $S(z)$ and $S(z_{j})$ are approximately tangent at $x + re(z,z_{j})$, see Lemma 1.1 in \cite{Wo2}. Since $t \geq \delta^{1 - 3\eta}$, the right hand side of \eqref{form11a} is much smaller than $\delta^{\eta} \lessapprox \dist(H_{i}'(z),H_{j}'(z))$, which will yield 
\begin{equation}\label{form23} |e(z,z_j)-e(z,z_i)| \gtrsim \dist(H_{i}'(z),H_{j}'(z)) \gtrapprox \delta^{\eta}, \qquad 1 \leq i < j \leq 3. \end{equation} 
To prove \eqref{form11a}, start by observing that $|\pi_{\theta_{j}}(z)-\pi_{\theta_{j}}(z_{j}) | \le \delta$ is equivalent to 
\begin{displaymath} \text{dist}(\sigma(z - z_{j}),\spa(\gamma(\theta_{j}))) \le \delta, \end{displaymath}
where $\spa(\gamma(\theta_{j})) = \spa(\cos \theta_{j},\sin \theta_{j},1)$. So, we can find $\alpha \in \R$ such that
\begin{equation}\label{form12a} |\sigma(x_{j} - x) + \alpha (\cos \theta_{j},\sin \theta_{j})| \leq |\sigma(z_{j} - z) + \alpha (\cos \theta_{j}, \sin \theta_{j}, 1)| \leq \delta. \end{equation}
We infer from the second inequality of \eqref{form12a} that $||r_{j} - r| - \alpha| = |\sigma(r_{j} - r) + \alpha| \leq \delta$, and consequently
\begin{displaymath} \alpha \sim |r - r_{j}| \sim t, \end{displaymath}
using $||x - x_{j}| - |r - r_{j}|| = \Delta(z,z_{j}) \leq 2\delta \ll t$ and $|x - x_{j}| + |r - r_{j}| \sim t$. In particular, $\alpha > 0$. Then, we estimate as follows:
\begin{align*} |e(z,z_{j}) + (\cos \theta_{j},\sin \theta_{j})| & = \left|\sigma\frac{(x_{j} - x)}{|x_{j} - x|} + (\cos \theta_{j},\sin \theta_{j}) \right|\\
& \leq \left|\sigma\frac{(x_{j} - x)}{|x_{j} - x|} - \sigma\frac{(x_{j} - x)}{\alpha} \right| + \left|\sigma\frac{(x_{j} - x)}{\alpha} + (\cos \theta_{j},\sin \theta_{j}) \right|\\
& \leq \frac{||x_{j} - x| - \alpha|}{\alpha} + \frac{1}{\alpha} \left|\sigma(x_{j} - x) + \alpha(\cos \theta_{j},\sin \theta_{j}) \right| \lesssim \frac{\delta}{t}, \end{align*} 
using \eqref{form12a} in the last estimate. This proves \eqref{form11a} and, hence, \eqref{form23}.

With \eqref{form23} in hand, it remains to apply a version of Marstrand's three circles lemma, for instance the one presented in Wolff's survey \cite[Lemma 3.2]{Wo}. The lemma implies that the set of points $z = (x,r) \in \B_{0}$ with $\Delta(z , z_{j}) \leq 2\delta$, $|z - z_{j}| \geq t$, and $|e(z,z_{i}) - e(z,z_{j})| \gtrapprox \delta^{\eta}$ for all $1 \leq i < j \leq 3$, is contained in the union of $\lesssim 1$ sets of diameter $\lessapprox \delta^{1 - 2\eta}$.\footnote{The statement of the lemma looks a bit different, but if you take a look at the first few lines of the proof, this is precisely what is done.} Consequently, for $(z_{1},z_{2},z_{3}) \in \B_{0}^{3}$ fixed, one has
\begin{displaymath} \mu(\{z \in Z' : z_{j} \asymp_{j} z \text{ for all } 1 \leq j \leq 3\}) \lessapprox (\delta^{1 - 2\eta})^{s} \leq \delta^{s - 6\eta}. \end{displaymath}
Combining this information with \eqref{form9a} gives the upper bound
\begin{displaymath} \mu(\{(z,z_{1},z_{2},z_{3}) \in Z' \times \B_{0}^{3} : z_{j} \asymp_{j} z \text{ for all } 1 \leq j \leq 3\}) \lessapprox t^{2s} \cdot \delta^{s - 6\eta}. \end{displaymath}
Comparing the upper bound with the lower bound \eqref{form4a} gives the relation
\begin{displaymath} \delta^{-\kappa} \lessapprox \delta^{-13\eta/3} \cdot \left(\frac{t}{\delta} \right)^{2s/3 - 1}. \end{displaymath} 
Recalling that $t \in [\delta,1]$ and $\kappa > \max\{0,2s/3 - 1\}$, the inequality above gives a lower bound for $\eta = \eta(s,\kappa) > 0$. This completes the proof of the lemma. \end{proof}

\subsection{From multiplicity estimates to lower bounds for Hausdorff dimension} We will deduce Theorem \ref{main} from Lemma \ref{mainLemma} and following general link between the multiplicity function $m_{\pi}^{\delta}$, and Hausdorff dimension:

\begin{lemma}\label{multToDimension} Fix $s \geq 0$, and assume that $\mu$ is a Borel probability measure on $\R^{3}$ with the following property: there are parameters $\eta > 0, \delta_{0} > 0$ such that the set of points $Z = Z_{\delta} \subset \R^{3}$ satisfying 
\begin{displaymath} \calH^{1}(\{\theta \in [0,2\pi) : m_{\pi}^{\delta}(\pi_{\theta}(z)) \geq \delta^{s}\}) \geq \delta^{\eta} \end{displaymath}
has measure $\mu(Z) \leq \delta^{\eta}$ for all $0 < \delta \leq \delta_{0}$. Then,
\begin{displaymath} \Hd \pi_{\theta}(\spt \mu) \geq s \quad \text{for $\calH^{1}$ almost every $\theta \in [0,2\pi)$}. \end{displaymath}
\end{lemma}

The argument used to prove Lemma \ref{multToDimension} is virtually the same as the proof of \cite[Theorem 7.2]{KOV}, but we give the details for completeness and convenience. 
\begin{proof}[Proof of Lemma \ref{multToDimension}] Write $K := \spt \mu \subset \R^{3}$. Assume, to reach a contradiction, that the set  
\begin{displaymath} E := \{\theta \in [0,2\pi) : \Hd \pi_{\theta}(K) < t\} \end{displaymath} 
has positive $\calH^{1}$ measure for some $0 < t < s$. For every $\theta \in E$, the set $\pi_{\theta}(K)$ can be covered by a family $\calD_{\theta}$ of discs on $V_{\theta}$ of dyadic radii $\leq 2^{-k_{0}} \leq \delta_{0}/3$, with the following properties:
\begin{itemize}
\item[(i)] $\pi_{\theta}(K) \subset \bigcup_{D \in \calD_{\theta}} D$,
\item[(ii)] $\sum_{D \in \calD_{\theta}} r(D)^{t} \leq 1$.
\end{itemize} 
Here $r(D)$ stands for the radius of $D$. For $\theta \in E$ fixed, write further
\begin{displaymath} \calD_{\theta}^{k} := \{D \in \calD_{\theta} : r(D) = 2^{-k}\}, \qquad k \geq k_{0}. \end{displaymath}
If $\theta \in E$ is fixed, and $\calD' \subset \calD_{\theta}$ is any sub-collection, we generally use the notation $D' := \cup \calD'$; in particular this definition applies to $D_{\theta}$ and $D_{\theta}^{k}$. Individual discs are denoted by simply $D$. We start by observing that
\begin{displaymath} \calH^{1}(E) = \int_{E} \mu(\pi^{-1}_{\theta}(D_{\theta})) \, d\theta \leq \sum_{k \geq k_{0}} \int_{E} \mu(\pi_{\theta}^{-1}(D_{\theta}^{k})) \, d\theta. \end{displaymath}
Treating $\calH^{1}(E) > 0$ as an absolute constant in the notation below, it follows that there exists $k \geq k_{0}$ such that
\begin{equation}\label{form12} \int_{E} \mu(\pi_{\theta}^{-1}(D_{\theta}^{k})) \, d\theta \gtrsim \frac{1}{k^{2}}. \end{equation}
Write $\delta := 2^{-k}$ for this $k \geq k_{0}$, so that $1/k^{2} = \log^{-2}(1/\delta) \approx 1$. We infer from \eqref{form12} that there exists a subset $E_{\delta} \subset E$ of length $\calH^{1}(E_{\delta}) \gtrapprox 1$ such that
\begin{equation}\label{form13} \mu(\pi_{\theta}^{-1}(D_{\theta}^{k})) \gtrapprox 1 \end{equation}
for $\theta \in E_{\delta}$. We replace $E$ by the subset $E_{\delta}$ without altering notation. 

Fix $\theta \in E$, and let $\calD_{\theta}^{k,j}$ consist of those discs $D \in \calD_{\theta}^{k}$ with the property that
\begin{displaymath} 2^{-j - 1} \leq \mu(\pi_{\theta}^{-1}(D)) \leq 2^{-j}. \end{displaymath}
Then
\begin{displaymath} 1 \lessapprox \mu(\pi_{\theta}^{-1}(D_{\theta}^{k})) \leq \sum_{j \geq 0} \mu(\pi^{-1}(D_{\theta}^{k,j})), \end{displaymath}
so there exists $j = j_{\theta} \geq 0$ such that
\begin{equation}\label{form15a} \mu(\pi_{\theta}^{-1}(D_{\theta}^{k,j})) \gtrapprox \frac{1}{j^{2}}. \end{equation} 
Using (ii), we can estimate
\begin{displaymath} \frac{1}{j^{2}} \lessapprox \mu(\pi_{\theta}^{-1}(D_{\theta}^{k,j})) \leq \sum_{D \in \calD_{\theta}^{k,j}} \mu(\pi_{\theta}^{-1}(D)) \leq |\calD_{\theta}^{k}| \cdot 2^{-j} \leq \delta^{-t}2^{-j}. \end{displaymath} 
This implies by \eqref{form15a}, and the choice $0 < t < s$, that
\begin{equation}\label{form14a} \mu(\pi_{\theta}^{-1}(D_{\theta}^{k,j})) \gtrapprox 1 \quad \text{and} \quad 2^{-j} > 2\delta^{s} \end{equation}
for $\theta \in E$, assuming that $\delta$ is sufficiently small. The index $j$ above depends on the choice of $\theta \in E$, but the inequality $j^{2}2^{-j} \gtrapprox \delta^{t}$ implies that there are only $\approx 1$ possible choices for $j$, so we may, if necessary, pass to a further subset $E' \subset E$ of size $\calH^{1}(E') \approx \calH^{1}(E)$ such that $j$ is fixed for $\theta \in E'$. We do so, but we keep denoting $E'$ by $E$.

Now, apply the main assumption of the lemma at scale $2\delta \leq \delta_{0}$. The conclusion is that the set $Z$ of those $z \in \R^{3}$ with
\begin{equation}\label{form16a} \calH^{1}(\{\theta \in [0,2\pi) : m^{2\delta}_{\pi}(\pi_{\theta}(z)) \geq \delta^{s}\}) \geq \delta^{\eta} \end{equation}
has measure $\mu(Z) \leq \delta^{\eta}$. In particular, using \eqref{form14a},
\begin{displaymath} 1 \lessapprox \mu(\pi_{\theta}^{-1}(D_{\theta}^{k,j})) \leq \mu(\pi_{\theta}^{-1}(D_{\theta}^{k,j}) \setminus Z) + \mu(Z) \end{displaymath}
for $\theta \in E$, which implies that
\begin{displaymath} \mu(\pi_{\theta}^{-1}(D_{\theta}^{k,j}) \setminus Z) \gtrapprox 1 \end{displaymath}
for $\theta \in E$. Using Fubini's theorem, we infer that
\begin{displaymath} 1 \lessapprox \int_{E} \mu(\pi^{-1}_{\theta}(D_{\theta}^{k,j}) \setminus Z) \, d\theta = \int_{K \setminus Z} \calH^{1}\left(\left\{\theta \in E : \pi_{\theta}(z) \in D_{\theta}^{k,j}\right\}\right) \, d\mu z. \end{displaymath} 
This implies the existence of $z \in K \setminus Z$ with $\calH^{1}(H(z)) \gtrapprox 1$, where
\begin{displaymath} H(z) := \left\{\theta \in E : \pi_{\theta}(z) \in D_{\theta}^{k,j}\right\}. \end{displaymath}
This contains the contradiction. In short, the idea is that since $z \in K \setminus Z$, the reverse inequality to \eqref{form16a} holds for $z$. So, there ought to be only few values of $\theta \in E$ such that $m_{\pi}^{2\delta}(\pi_{\theta}(z))$ is high. Since $\calH^{1}(H(z)) \gtrapprox 1$ is much larger than $\delta^{\eta}$, it remains to show that $m_{\pi}^{2\delta}(\pi_{\theta}(z))$ is high whenever $\theta \in H(z)$.

Fix $\theta \in H(z)$, so $\pi_{\theta}(z) \in D$ for some $D \in \calD_{\theta}^{k,j}$. By the definition of $\calD_{\theta}^{k,j}$, and \eqref{form14a}, this means that
\begin{displaymath} \mu(\pi_{\theta}^{-1}(D)) \geq 2^{-j - 1} > \delta^{s}. \end{displaymath}
It follows that 
\begin{displaymath} m^{2\delta}_{\pi}(\pi_{\theta}(z)) = \mu(\pi_{\theta}^{-1}(D(\pi_{\theta}(z),2\delta)) \geq \mu(\pi_{\theta}^{-1}(D)) > \delta^{s} \end{displaymath}
for all $\theta \in H(z)$. Since $\calH^{1}(H(z)) \gtrapprox 1$ is far larger than $\delta^{\eta}$, this contradicts $z \in K \setminus Z$. The proof of Lemma \ref{multToDimension} is complete.
\end{proof}

\begin{proof}[Proof of Theorem \ref{main}] Fix a Borel set $K \subset \B_{0}$, and assume without loss of generality that $\Hd K > 0$. Pick $0 < s < \Hd K$, and $\kappa > \max\{0,\tfrac{2s}{3} - 1\}$. Let $\mu$ be a Borel probability measure with $\spt \mu \subset K$, and $\mu(B(z,r)) \lesssim r^{s}$ for all balls $B(z,r) \subset \R^{3}$. By Lemma \ref{mainLemma}, the set $Z = Z_{\delta}$ of points $z \in \B_{0}$ satisfying
\begin{displaymath} \calH^{1}(\{\theta \in [0,2\pi) : m_{\pi}^{\delta}(\pi_{\theta}(z)) \geq \delta^{s - \kappa}\}) \geq \delta^{\eta} \end{displaymath}
has measure $\mu(Z) \leq \delta^{\eta}$ for all $\delta > 0$ sufficiently small, and for some $\eta > 0$. Then, Lemma \ref{multToDimension} implies that
\begin{displaymath} \Hd \pi_{\theta}(K) \geq s - \kappa \quad \text{ for $\calH^{1}$ almost every $\theta \in [0,2\pi)$.} \end{displaymath}
Theorem \ref{main} follows by letting $\kappa \searrow \max\{0,\tfrac{2s}{3} - 1\}$, and then $s \nearrow \Hd K$. \end{proof}

\appendix

\section{Reduction to the plane $W$}\label{appendixA}

Recall that $W_{t} = \{(x,y,r) \in \R^{3} : r = t\}$ for $t \in (-1,1) \setminus \{0\}$, and $S_{W_{t}} = S^{2} \cap W_{t}$. We prove the following lemma from Section \ref{mainProof}:

\begin{lemma} Let $t \in (-1,1) \setminus \{0\}$, and parametrise $S_{W_{t}}$ by the curve $\gamma_{t} \colon [0,2\pi) \to S^{2}$,
\begin{displaymath} \gamma_{t}(\theta) = (\sqrt{1 - t^{2}}\cos \theta, \sqrt{1 - t^{2}} \sin \theta, t). \end{displaymath}
Write $V_{\theta}^{t} := \spa(\gamma_{t}(\theta)^{\perp})$ and $V_{\theta} := V_{\theta}^{1/\sqrt{2}}$. Finally, define 
\begin{displaymath} \pi_{\theta}^{t} := \pi_{V_{\theta}^{t}} \quad \text{and} \quad \pi_{\theta} := \pi_{V_{\theta}}. \end{displaymath}
Then, there exists an invertible linear map $B_{t} \colon \R^{3} \to \R^{3}$, depending only on $t$, and a family of invertible linear maps $A_{\theta}^{t} \colon V_{\theta} \to V_{\theta}^{t}$, depending on both $\theta$ and $t$, such that the following relation holds:
\begin{displaymath} \pi_{\theta}^{t} = A_{\theta}^{t} \circ \pi_{\theta} \circ B_{t}, \qquad \theta \in [0,2\pi). \end{displaymath} 
\end{lemma}

\begin{proof} An orthonormal basis of $V_{\theta}^{t}$ is given by
\begin{displaymath} \{e_{\theta,1}^{t},e_{\theta,2}^{t}\} = \left\{(-\sin \theta, \cos \theta, 0), \left( t\cos \theta , t\sin \theta , -\sqrt{1 - t^{2}} \right) \right\}. \end{displaymath}
Note that the first vector is independent of $t$. The orthogonal projection to $V_{\theta}^{t}$ can be written as
\begin{align*} \pi_{\theta}^{t}(x,y,r) & = [(x,y,r) \cdot e_{\theta,1}^{t}]e_{\theta,1}^{t} + \left[(x,y,r) \cdot \left( t\cos \theta , t\sin \theta , -\sqrt{1 - t^{2}} \right) \right]e_{\theta,2}^{t}\\
& = [B_{t}(x,y,r) \cdot e_{\theta,1}^{1/\sqrt{2}}]e_{\theta,1}^{t} + [B_{t}(x,y,r) \cdot e_{\theta,2}^{1/\sqrt{2}}]t\sqrt{2}e_{\theta,2}^{t}\\
& = A_{\theta}^{t}[\pi_{\theta}(B_{t}(x,y,r))], \qquad (x,y,r) \in \R^{3},  \end{align*}
where $B_{t}$ is the linear map $B_{t}(x,y,r) = (x,y,r\sqrt{1 - t^{2}}/t)$, and $A_{\theta}^{t} \colon V_{\theta} \to V_{\theta}^{t}$ is the linear map determined by $A_{\theta}^{t}(e_{\theta,1}^{1/\sqrt{2}}) = e_{\theta,1}^{t}$ and $A_{\theta}^{t}(e_{\theta,2}^{1/\sqrt{2}}) = t\sqrt{2}e_{\theta,2}^{t}$. This concludes the proof. \end{proof}

\end{document}